\documentclass[11pt]{amsart}

\usepackage{amsmath,amssymb,amsthm,mathtools}
\usepackage{enumitem}
\usepackage{microtype}
\usepackage[hidelinks]{hyperref}

\newtheorem{theorem}{Theorem}[section]
\newtheorem{proposition}[theorem]{Proposition}
\newtheorem{lemma}[theorem]{Lemma}

\theoremstyle{definition}

\theoremstyle{remark}
\newtheorem{remark}[theorem]{Remark}

\newcommand{\F}{\mathbb F}
\newcommand{\PP}{\mathbb P}
\newcommand{\ZZ}{\mathbb Z}
\newcommand{\Gal}{\operatorname{Gal}}
\newcommand{\Aut}{\operatorname{Aut}}
\newcommand{\Div}{\operatorname{div}}
\newcommand{\cL}{\mathcal L}
\newcommand{\cO}{\mathcal O}
\newcommand{\Sm}{\mathcal S}
\newcommand{\Rm}{\mathcal R}
\newcommand{\tSm}{\widetilde{\mathcal S}}
\newcommand{\tRm}{\widetilde{\mathcal R}}

\title[On Skabelund's ray class field covers]
      {On Skabelund's Ray Class Field Covers of the Suzuki and Ree Curves}

\author{Saeed Tafazolian}
\address{Departamento de Matem\'atica, IMECC, Universidade Estadual de Campinas,
Rua S\'ergio Buarque de Holanda 651, Campinas, SP 13083-859, Brazil}
\email{saeed@unicamp.br}


\begin{document}

\begin{abstract}
Let $\Sm_q$ and $\Rm_q$ denote the Suzuki and Ree curves.  Motivated by
the Giulietti--Korchm\'aros curve, Skabelund constructed cyclic covers
$\tSm_q$ and $\tRm_q$ of these curves and proved that they are maximal
over $\F_{q^4}$ and $\F_{q^6}$, respectively.  In the same paper he
associated to the Suzuki and Ree curves certain ray class field covers
$\Sm_{\rm rcf}$ and $\Rm_{\rm rcf}$, and showed that there are towers
\[
   \Sm_{\rm rcf}\longrightarrow\tSm_q\longrightarrow\Sm_q,
   \qquad
   \Rm_{\rm rcf}\longrightarrow\tRm_q\longrightarrow\Rm_q .
\]
Computations for small values of $q$ suggested that the first map in
each tower is always an isomorphism, and the general case was left open.
We prove that
\[
   \Sm_{\rm rcf}=\tSm_q,\qquad \Rm_{\rm rcf}=\tRm_q
\]
for every admissible $q$, the comparison being made over $\F_{q^4}$ in
the Suzuki case and over $\F_{q^6}$ in the Ree case.  The proof combines
a Kummer normal form of the ray class extension, allowing a constant
twist, with the centrality of its Galois group among lifted
automorphisms, the standard involution of the base curve, and the first
positive non-gap at the rational point at infinity.
\end{abstract}

\maketitle

\section{Introduction}

Let $Y$ be a smooth projective curve of genus $g$ defined over a finite
field $\F_\ell$. We say that $Y$ is maximal over $\F_\ell$ if
\[
        \#Y(\F_\ell)=\ell+1+2g\sqrt{\ell}.
\]
The Hermitian curve is the standard example. Two other important
examples are the Suzuki and Ree curves, which become maximal after
suitable extensions of the ground field. Together with the Hermitian
curve, they are the Deligne--Lusztig curves associated with the twisted
groups of types ${}^2A_2$, ${}^2B_2$, and ${}^2G_2$; see \cite{HKT}.

Ray class fields occur naturally in this setting. Lauter \cite{Lauter}
gave a ray class field description of the Deligne--Lusztig curves.
Later, Giulietti and Korchm\'aros \cite{GK} constructed a new maximal
curve as a cyclic cover of the Hermitian curve. Motivated by this
example, Skabelund \cite{Skabelund} constructed analogous cyclic covers
of the Suzuki and Ree curves.

Let $\Sm_q$ and $\Rm_q$ denote the Suzuki and Ree curves, with
\[
       q=2q_0^2=2^{2s+1}
       \qquad\text{and}\qquad
       q=3q_0^2=3^{2s+1},
       \qquad s\ge1,
\]
respectively. Put
\[
       m=q-2q_0+1
\]
in the Suzuki case and
\[
       m=q-3q_0+1
\]
in the Ree case. Skabelund's covers $\tSm_q\to\Sm_q$ and
$\tRm_q\to\Rm_q$ are obtained by adjoining a root of
\[
       t^m=x^q-x
\]
to the function field of the base curve.  He proved that $\tSm_q$ is maximal over $\F_{q^4}$ and
that $\tRm_q$ is maximal over $\F_{q^6}$.

Skabelund also considered ray class field covers of the Suzuki and Ree
curves over $\F_{q^4}$ and $\F_{q^6}$, respectively, in which the
$\F_q$-rational points are allowed to ramify and the closed points of
degree $4$, respectively $6$, are required to split completely. We
denote these curves by $\Sm_{\rm rcf}$ and $\Rm_{\rm rcf}$; see
Section~\ref{sec:prelim}. For the Hermitian curve the analogous
construction gives exactly the Giulietti--Korchm\'aros curve
\cite[Cor.~5.5]{Skabelund}. For the Suzuki and Ree curves one has
\[
        \Sm_{\rm rcf}\longrightarrow\tSm_q\longrightarrow\Sm_q,
        \qquad
        \Rm_{\rm rcf}\longrightarrow\tRm_q\longrightarrow\Rm_q.
\]
Skabelund proved that the ray class field curve is maximal and gave a
Kummer description of its function field \cite[Thm.~5.2 and
Cor.~5.3]{Skabelund}. The divisibility condition appearing there is
deduced using \cite{GarciaTafazolian}. We shall not use this precise
description. Instead, in Section~\ref{sec:action} we derive directly the
weaker form needed here, allowing a constant Kummer twist; in
particular, the divisibility relation used below is obtained
independently.
By computations in Magma, Skabelund verified that
$\Sm_{\rm rcf}=\tSm_q$ for $q=2^{2s+1}$, $1\le s\le6$, and that
$\Rm_{\rm rcf}=\tRm_q$ for $q=27$. He left open whether these
equalities hold for every admissible $q$ \cite[p.~538]{Skabelund}.

The aim of this paper is to prove that they do. When the Skabelund
cover is compared with the ray class field curve, we work over
$\F_{q^4}$ in the Suzuki case and over $\F_{q^6}$ in the Ree case.

\begin{theorem}\label{thm:main}
Let $q=2^{2s+1}$, $s\ge1$. Then the ray class field cover of the Suzuki
curve considered in \cite{Skabelund} is the Skabelund cover:
\[
        \Sm_{\rm rcf}=\tSm_q.
\]
Let $q=3^{2s+1}$, $s\ge1$. Then the corresponding statement holds for
the Ree curve:
\[
        \Rm_{\rm rcf}=\tRm_q.
\]
In the first equality the function fields are compared over $\F_{q^4}$,
and in the second over $\F_{q^6}$.
\end{theorem}

The full automorphism groups of $\tSm_q$ and $\tRm_q$ were determined in
\cite{GMQZ}; hence Theorem~\ref{thm:main} also identifies the full
automorphism groups of the corresponding ray class field curves.

It is worth pointing out why the argument used for the Hermitian curve
in \cite[Prop.~5.4]{Skabelund} does not settle the two cases above. In
both the Suzuki and Ree cases there are proper multiples $r$ of $m$ for
which the Kummer curve
\[
        u^r=x^q-x
\]
is still maximal; see \cite[p.~538]{Skabelund} and
Proposition~\ref{prop:larger-maximal-kummer}. Thus maximality alone
cannot force the ray class field to coincide with the Skabelund cover.

We indicate the main point of the proof. Let $d=4$ in the Suzuki case
and $d=6$ in the Ree case, and put
\[
   K=\F_{q^d}(X),\qquad
   L=\F_{q^d}(X_{\rm rcf}),\qquad
   n=[L:K],
\]
and
\[
   N=\#X(\F_q)=q^{d/2}+1,\qquad
   s_0=\frac{N}{m}.
\]
The Skabelund field is contained in $L$, so $m\mid n$. We show that
$L/K$ is cyclic of degree prime to the characteristic, that the lifts
to $L$ of the automorphisms needed below centralize $\Gal(L/K)$, and
that
\[
       L=K(v),\qquad v^n=\lambda(x^q-x)
\]
for some $\lambda\in\F_{q^d}^{\times}$. Let $\iota$ be the standard
involution of the Suzuki or Ree curve, interchanging $P_\infty$ with a
rational point $P_0$. Since a lift of $\iota$ centralizes
$\Gal(L/K)$, it preserves the relevant Kummer eigenspace, and hence
\[
       \iota(v)=hv
\]
for some $h\in K^\times$. Since the lift fixes the constant field,
taking $n$-th powers and comparing divisors gives
\[
       n\,\Div(h)=N(P_\infty-P_0).
\]
It follows that $n\mid N$ and that $N/n$ is a positive non-gap at
$P_\infty$. Writing $n=mk$, we obtain
\[
       \frac{N}{n}=\frac{s_0}{k}.
\]
If $k>1$, this is smaller than the first positive non-gap at
$P_\infty$, a contradiction. Hence $k=1$, and the ray class field
curve coincides with the Skabelund cover.

Section~\ref{sec:prelim} recalls the curves, the Skabelund covers, the
ray class construction, and two elementary facts concerning
Weierstrass non-gaps. In Section~\ref{sec:action} we recall the
relevant ray-class results of Skabelund and derive the additional
structural properties needed for the argument. The standard involution
is used in Section~\ref{sec:involution} to produce the decisive
non-gap, and the proof of Theorem~\ref{thm:main} is completed in
Section~\ref{sec:proof}. The last section explains why maximality of
auxiliary Kummer quotients alone is insufficient.
\section{Preliminaries}\label{sec:prelim}

\subsection{The Suzuki and Ree curves}

We recall only the facts that will be used later. For the general
theory of these curves we refer to \cite{HKT} and the references
therein.

Let
\[
       q=2q_0^2=2^{2s+1},\qquad s\ge1.
\]
The Suzuki curve $\Sm_q/\F_q$ has an affine equation
\begin{equation}\label{eq:suzuki}
       y^q+y=x^{q_0}(x^q+x).
\end{equation}
It has genus $q_0(q-1)$ and
\begin{equation}\label{eq:suzuki-points}
       \#\Sm_q(\F_q)=q^2+1.
\end{equation}
There is a unique point $P_\infty$ at infinity and
\begin{equation}\label{eq:suzuki-x-pole}
       (x)_\infty=qP_\infty.
\end{equation}
The curve is maximal over $\F_{q^4}$.

For
\[
       q=3q_0^2=3^{2s+1},\qquad s\ge1,
\]
the Ree curve $\Rm_q/\F_q$ is given by
\begin{equation}\label{eq:ree}
 \begin{cases}
       y^q-y=x^{q_0}(x^q-x),\\
       z^q-z=x^{2q_0}(x^q-x).
 \end{cases}
\end{equation}
Its genus is
\[
       \frac32 q_0(q-1)(q+q_0+1),
\]
and
\begin{equation}\label{eq:ree-points}
       \#\Rm_q(\F_q)=q^3+1.
\end{equation}
Again there is a unique point $P_\infty$ at infinity, now with
\begin{equation}\label{eq:ree-x-pole}
       (x)_\infty=q^2P_\infty.
\end{equation}
The Ree curve is maximal over $\F_{q^6}$.

It is convenient to use the following notation throughout the paper.
Put
\begin{equation}\label{eq:dN}
 d=\begin{cases}
     4,&X=\Sm_q,\\
     6,&X=\Rm_q,
   \end{cases}
 \qquad
 N=\#X(\F_q)=q^{d/2}+1,
\end{equation}
and
\begin{equation}\label{eq:fdef}
       f=x^q-x.
\end{equation}

\subsection{The Skabelund covers}

Set
\begin{equation}\label{eq:mdef}
 m=\begin{cases}
     q-2q_0+1,&X=\Sm_q,\\
     q-3q_0+1,&X=\Rm_q.
   \end{cases}
\end{equation}
The cover $\widetilde X\to X$ introduced in \cite{Skabelund} is defined
by
\begin{equation}\label{eq:skcover}
       t^m=f.
\end{equation}
Since $m$ is prime to the characteristic, this is a tame Kummer cover.
In the Suzuki case it is $\tSm_q\to\Sm_q$ and in the Ree case it is
$\tRm_q\to\Rm_q$. Skabelund proved \cite[Thms.~3.1 and~4.1]{Skabelund}
\begin{equation}\label{eq:max-sk}
       \tSm_q\text{ is maximal over }\F_{q^4},
       \qquad
       \tRm_q\text{ is maximal over }\F_{q^6}.
\end{equation}

The two elementary factorizations
\begin{align}
 q^2+1
   &=(q-2q_0+1)(q+2q_0+1),
       &&q=2q_0^2,\label{eq:factor-suzuki}\\
 q^3+1
   &=(q-3q_0+1)
     (q^2+3qq_0+2q+3q_0+1),
       &&q=3q_0^2,\label{eq:factor-ree}
\end{align}
show that $m\mid N$ in both cases. We shall write
\begin{equation}\label{eq:s0def}
       s_0=\frac Nm.
\end{equation}
Thus
\begin{equation}\label{eq:s0-values}
 s_0=\begin{cases}
       q+2q_0+1,&X=\Sm_q,\\
       q^2+3qq_0+2q+3q_0+1,&X=\Rm_q.
     \end{cases}
\end{equation}

\subsection{The ray class field in the present situation}

Put
\[
       Q=q^d,\qquad R=q^{d/2},\qquad K=\F_Q(X),
\]
so that $Q=R^2$ and $N=R+1$. We view the $\F_q$-rational points of $X$
as $\F_Q$-rational places of $K$ and put
\begin{equation}\label{eq:modulus}
       \mathfrak m=\sum_{P\in X(\F_q)}P.
\end{equation}
Let $\Sigma$ be the set of places of $K$ lying over the closed points
of degree $d$ of $X/\F_q$. These points form the second short orbit in
the Suzuki and Ree cases; see \cite[Props.~2.2 and~2.3]{Skabelund}.

Following \cite[Sec.~5]{Skabelund}, let $L$ be the maximal abelian
extension of $K$, inside a fixed algebraic closure, of conductor
dividing $\mathfrak m$ in which all places of $\Sigma$ split completely,
and write
\[
       L=\F_Q(X_{\rm rcf}).
\]
See Remark~\ref{rem:conductor} for the comparison with the phrase
``of conductor $\mathfrak m$'' used in \cite{Skabelund}.

\begin{lemma}\label{lem:splitting-set}
One has
\[
       \Sigma=X(\F_Q)\setminus X(\F_q).
\]
Consequently, if $S=|\Sigma|$, then
\begin{equation}\label{eq:S}
       S=Q+1+2g(X)R-N
        =R\bigl(R-1+2g(X)\bigr)>0.
\end{equation}
\end{lemma}

\begin{proof}
Every point of $X(\F_Q)$ comes from a closed point of $X/\F_q$ whose
degree divides $d$. It therefore suffices to show that the Suzuki curve
has no points of degree $2$, and that the Ree curve has no points of
degree $2$ or $3$. Since $P_\infty$ is $\F_q$-rational, it is enough
to consider affine points.

Let $P=(x,y)\in\Sm_q(\F_{q^2})$ and put $f=x^q-x$. Then
$f^q=x-x^q=f$, since the characteristic is $2$. Raising
\eqref{eq:suzuki} to the $q$-th power gives
\[
       y+y^q=(x^q)^{q_0}f.
\]
Comparison with \eqref{eq:suzuki} yields
\[
       0=f\bigl(x^{q_0}+(x^q)^{q_0}\bigr)
        =f\cdot f^{q_0}
        =f^{q_0+1}.
\]
Hence $x\in\F_q$, and then $y^q+y=0$ gives $y\in\F_q$.

Let $P=(x,y,z)\in\Rm_q(\F_{q^2})$. Now $f^q=-f$, and the same
computation with the first equation of \eqref{eq:ree} gives
\[
       f\bigl(x^{q_0}-(x^q)^{q_0}\bigr)=0.
\]
Since $q_0$ is odd,
\[
       x^{q_0}-(x^q)^{q_0}
       =(x-x^q)^{q_0}
       =-f^{q_0},
\]
so again $f^{q_0+1}=0$. Thus $x\in\F_q$, and the two equations of
\eqref{eq:ree} give $y,z\in\F_q$.

Finally, let $P=(x,y,z)\in\Rm_q(\F_{q^3})$ and suppose that
$x\notin\F_q$. For $j\in\ZZ/3\ZZ$ put
\[
       x_j=x^{q^j},\qquad
       a_j=x_j^{q_0},\qquad
       f_j=x_{j+1}-x_j.
\]
Applying the $q^j$-th power map to the two equations of
\eqref{eq:ree} and summing over $j$, the left-hand sides telescope to
zero. Together with
\[
       \sum_j f_j=x^{q^3}-x=0,
\]
this gives
\[
       \sum_j f_j=0,\qquad
       \sum_j a_jf_j=0,\qquad
       \sum_j a_j^2f_j=0.
\]
Since $x$ has degree $3$ over $\F_q$, the $x_j$ are distinct, and hence
so are the $a_j$. The Vandermonde matrix
\[
       (a_j^i)_{0\le i\le2}
\]
is therefore nonsingular, so all $f_j$ vanish. Thus $x^q=x$, a
contradiction. Hence $x\in\F_q$, and as before $y,z\in\F_q$.

This proves the first assertion. Formula \eqref{eq:S} follows from the
maximality of $X$ over $\F_Q$ and $N=R+1$.
\end{proof}

In particular, $\Sigma$ is nonempty. By class field theory, $L/K$ is
finite. Since the places in $\Sigma$ are $\F_Q$-rational and split
completely in $L$, the full constant field of $L$ is $\F_Q$.

\subsection{Weierstrass non-gaps and extension of constants}

Let $Y$ be a smooth projective curve over a field $k$ and let
$P\in Y(k)$.  For $a\ge0$, write
\[
       \ell_k(aP)=\dim_k\cL(aP).
\]
The Weierstrass semigroup at $P$ is
\[
       H_Y(P)=\{0\}\cup
       \{a\ge1:\ell_k(aP)>\ell_k((a-1)P)\}.
\]
Equivalently, $a>0$ belongs to $H_Y(P)$ if and only if there is a
rational function on $Y$ whose only pole is $P$ and whose pole order
there is $a$.

We shall also use the following elementary fact about extension of
constants.

\begin{lemma}\label{lem:basechange-semigroup}
Let $k'/k$ be a field extension, let $Y/k$ be a smooth projective
curve, and let $P\in Y(k)$.  Then
\[
       H_Y(P)=H_{Y_{k'}}(P).
\]
\end{lemma}

\begin{proof}
For every $a\ge0$, flat base change gives
\[
 H^0\bigl(Y_{k'},\cO_{Y_{k'}}(aP)\bigr)
   \simeq
 H^0\bigl(Y,\cO_Y(aP)\bigr)\otimes_k k'.
\]
Hence
\[
       \ell_{k'}(aP)=\ell_k(aP)
\]
for every $a$.  The jumps of the Riemann--Roch spaces are therefore the
same over $k$ and $k'$.
\end{proof}

We shall also use the following standard bound of Lewittes
\cite{Lewittes}.

\begin{lemma}\label{lem:lewittes}
Let $Y/\F_q$ be a curve and $P\in Y(\F_q)$.  If $a$ is the first
positive non-gap at $P$, then
\begin{equation}\label{eq:lewittes}
       \#Y(\F_q)\le qa+1.
\end{equation}
\end{lemma}
\section{The ray extension and its automorphisms}\label{sec:action}

We begin with the divisor of the function $f=x^q-x$. Put
\begin{equation}\label{eq:Ddef}
       D=\mathfrak m=\sum_{P\in X(\F_q)}P.
\end{equation}
Thus $P_\infty$ occurs in $D$ with coefficient $1$.

\begin{lemma}\label{lem:divf}
For $X=\Sm_q$ or $X=\Rm_q$ one has
\begin{equation}\label{eq:divf}
       \Div(f)=D-NP_\infty.
\end{equation}
\end{lemma}

\begin{proof}
Suppose first that $X=\Sm_q$. Every affine $\F_q$-rational point is a
zero of $f$. By \eqref{eq:suzuki-x-pole}, the unique pole of $f$ is
$P_\infty$, of order $q^2$. Since there are exactly $q^2$ affine
rational points, all these zeros are simple and there are no others.
Thus
\[
       \Div(f)=D-(q^2+1)P_\infty.
\]
For the Ree curve, \eqref{eq:ree-x-pole} gives a unique pole of order
$q^3$, and the $q^3$ affine rational points are again all the zeros.
This gives \eqref{eq:divf} in both cases.
\end{proof}

We shall use the invariance of the ray class field under
$\F_q$-automorphisms of the base curve.

\begin{lemma}\label{lem:ray-invariant}
Let $\sigma\in\Aut_{\F_q}(X)$, extended trivially on $\F_Q$. Any
extension of $\sigma$ to a fixed algebraic closure of $K$ satisfies
\[
       \sigma(L)=L.
\]
\end{lemma}

\begin{proof}
Since $\sigma$ is defined over $\F_q$, it permutes $X(\F_q)$ and hence
preserves the modulus $\mathfrak m$. It also preserves the degrees of
closed points of $X/\F_q$, and therefore permutes the set $\Sigma$.

Conjugation by $\sigma$ carries the Galois group, ramification groups,
and decomposition groups of $L/K$ onto those of $\sigma(L)/K$. Hence
$\sigma(L)/K$ is abelian, its conductor satisfies
\[
       \mathfrak f(\sigma(L)/K)
       =\sigma\bigl(\mathfrak f(L/K)\bigr)
       \le \sigma(\mathfrak m)
       =\mathfrak m,
\]
and every place of $\sigma(\Sigma)=\Sigma$ splits completely in it.
By the defining maximality of $L$,
\[
       \sigma(L)\subseteq L.
\]
Applying the same argument to $\sigma^{-1}$ gives
$\sigma^{-1}(L)\subseteq L$, and therefore $L\subseteq\sigma(L)$.
Hence $\sigma(L)=L$.
\end{proof}

We next specify a small subgroup of automorphisms; no classification of
the full automorphism group is needed. In the Suzuki case the maps
\[
 (x,y)\longmapsto
 (x+a,\ y+a^{q_0}x+b),\qquad a,b\in\F_q,
\]
and in the Ree case the maps
\[
 (x,y,z)\longmapsto
 \bigl(x+a,\ y+a^{q_0}x+b,\
       z-a^{q_0}y+a^{2q_0}x+c\bigr),
 \qquad a,b,c\in\F_q,
\]
preserve the defining equations and fix $P_\infty$. They act
transitively on the affine $\F_q$-rational points. Together with the
standard involution, which exchanges $P_\infty$ with the affine origin,
they generate a finite subgroup
\[
       G\le\Aut_{\F_q}(X)
\]
acting doubly transitively on $X(\F_q)$. For the Suzuki involution see
\cite{HansenStichtenoth,Henn}, and for the Ree involution see
\cite{Pedersen}; cf.\ \cite[Lemmas~3.3 and~4.2]{Skabelund}.

We now collect the properties of the ray extension that will be used
below. The first assertion follows from \cite[Sec.~5]{Skabelund}; the
remaining assertions are proved here in the form needed below.

\begin{proposition}\label{prop:ray-structure}
Let
\[
       A=\Gal(L/K),\qquad n=[L:K],
\]
and let $p$ denote the characteristic. Then the following hold.
\begin{enumerate}[label=\rm(\roman*)]

\item\label{it:containment}
The Skabelund field $K(t)$, $t^m=f$, is contained in $L$. In
particular, $m\mid n$.

\item\label{it:cyclic}
The extension $L/K$ is cyclic, totally ramified at every point of $D$,
and unramified outside $D$. Moreover, $p\nmid n$ and $n\mid Q-1$.

\item\label{it:central}
If $\Gamma$ denotes the group of all lifts to $L$ of the elements of
$G$, then
\[
       1\longrightarrow A\longrightarrow\Gamma\longrightarrow G
       \longrightarrow1
\]
is exact and $A$ is central in $\Gamma$.

\item\label{it:normalform}
There are $\lambda\in\F_Q^\times$ and $v\in L$ such that
\begin{equation}\label{eq:twisted-normal-form}
       L=K(v),\qquad v^n=\lambda f.
\end{equation}

\end{enumerate}
\end{proposition}

\begin{proof}
Assertion \ref{it:containment} is the inclusion obtained in
\cite[Sec.~5]{Skabelund}; see also \cite[p.~536]{Skabelund}.

We prove \ref{it:cyclic} directly. Since the conductor divides the
reduced divisor $\mathfrak m$, the extension $L/K$ is unramified outside
$D$ and tamely ramified at the points of $D$. By Lemma~\ref{lem:ray-invariant}, every element of $G$ has a lift to
$L$. Since $G$ is transitive on $D$, the ramification index above every
point of $D$ is the same; denote it by $e$. Riemann--Hurwitz gives
\[
  2g(L)-2=n(2g(X)-2)+Nn\left(1-\frac1e\right).
\]
All $S$ places in $\Sigma$ split completely in $L$, and hence
$\#X_{\rm rcf}(\F_Q)\ge nS$. The Hasse--Weil bound, together with the
preceding formula and \eqref{eq:S}, gives
\[
 \#X_{\rm rcf}(\F_Q)
 \le nS+N\left(R+1-R\frac ne\right).
\]
Thus $R+1-Rn/e\ge0$. If $e<n$, then $n/e\ge2$, and hence
\[
 R+1-R\frac ne\le1-R<0,
\]
a contradiction. Therefore $e=n$.

Thus every point of $D$ is totally ramified. Let $P\in D$ and let $P'$ be
the unique point above it. Its residue field is $\F_Q$, and the inertia
group is all of $A$. Since the ramification is tame, the action on a
local parameter gives an injection
\[
       A\hookrightarrow\F_Q^\times
\]
by \cite[Prop.~3.8.5]{Stichtenoth}. Hence $A$ is cyclic, $p\nmid n$, and
$n\mid Q-1$. This proves \ref{it:cyclic}.

We prove \ref{it:central}. By Lemma~\ref{lem:ray-invariant}, every
element of $G$ has a lift to $L$, and the lifts of the identity are
exactly the elements of $A$. Hence the displayed sequence is exact.

Fix $P\in D$, and let $P'$ be the unique point above it. The inverse
image of the stabilizer $G_P$ is the stabilizer $\Gamma_{P'}$. Since
every element of $\Gamma$ fixes $\F_Q$ pointwise and the residue field
of $P'$ is $\F_Q$, the decomposition group $\Gamma_{P'}$ is also the
inertia group at $P'$. By \cite[Prop.~3.8.5]{Stichtenoth},
\[
       \Gamma_{P'}/\Gamma_{P'}^{(1)}
\]
embeds in $\F_Q^\times$, and $\Gamma_{P'}^{(1)}$ is a $p$-group. If
$b\in\Gamma_{P'}$ and $a\in A$, then $[b,a]\in A$ because $A$ is
normal, while
\[
       [b,a]\in\Gamma_{P'}^{(1)}
\]
because the quotient above is abelian. Since $|A|=n$ is prime to $p$,
we get $[b,a]=1$. Thus every $\Gamma_{P'}$ centralizes $A$.

It remains to note that the point stabilizers generate $G$. The
translation subgroup of $G_{P_\infty}$ is transitive on
$D\setminus\{P_\infty\}$, and the standard involution moves
$P_\infty$. Hence $G$ is doubly transitive on $D$, and in particular
primitive. Let $M$ be the normal closure of $G_{P_\infty}$ in $G$.
The $M$-orbits form a $G$-invariant partition of $D$. Since
$G_{P_\infty}$ contains translations which move affine rational
points, $M$ does not act trivially; hence, by primitivity, $M$ is
transitive. Given $\gamma\in G$, choose $\mu\in M$ such that
\[
       \mu(P_\infty)=\gamma(P_\infty).
\]
Then $\mu^{-1}\gamma\in G_{P_\infty}\subseteq M$, and therefore
$\gamma\in M$. Thus $G=M$, so the point stabilizers generate $G$.

The subgroup of $\Gamma$ generated by the groups $\Gamma_{P'}$,
$P\in D$, contains $A$ and maps onto $G$. It is therefore all of
$\Gamma$. Hence $A$ is central in $\Gamma$.

We finally prove \ref{it:normalform}. Put
\[
       F=\F_Q(x).
\]
The extension $K/F$ is Galois. Its group $U$ consists of the vertical
translations
\[
       (x,y)\longmapsto(x,y+b)
\]
in the Suzuki case, and
\[
       (x,y,z)\longmapsto(x,y+b,z+c)
\]
in the Ree case, with $b,c\in\F_q$. Indeed, these give $q$,
respectively $q^2$, automorphisms of $K/F$, and
\[
       [K:F]=q,\qquad [K:F]=q^2,
\]
respectively, by \eqref{eq:suzuki-x-pole} and
\eqref{eq:ree-x-pole}. The additive equations show that $K/F$ is
unramified at every finite place. Moreover, it splits completely above
each place $x=a$, $a\in\F_q$.

Let $\widehat U$ be the inverse image of $U$ in $\Gamma$. Then
\[
       |\widehat U|=n|U|=[L:F],
\]
and every element of $\widehat U$ fixes $F$. Hence $L/F$ is Galois
with group $\widehat U$. By \ref{it:central}, $A$ is central in
$\widehat U$. Since $A$ is a Hall $p'$-subgroup of $\widehat U$ and
$\widehat U/A\cong U$ is a $p$-group, Schur--Zassenhaus gives a
complement $U'$ to $A$. Any two such complements are conjugate by an
element of $A$, and $A$ is central; hence $U'$ is unique. Thus
\[
       \widehat U=A\times U'.
\]
In particular, $U'$ is the unique Sylow $p$-subgroup of $\widehat U$,
and hence is characteristic.

Let
\[
       E=L^{U'}.
\]
Then
\[
       \Gal(E/F)\cong A,\qquad
       E\cap K=F,\qquad
       EK=L.
\]
Let $\tau$ be a generator of $A$. Since $n\mid Q-1$, we have
$\mu_n\subset\F_Q^\times$. Kummer theory
\cite[Prop.~3.7.3]{Stichtenoth} gives
\[
       E=F(v),\qquad v^n=\rho
\]
for some $\rho\in F^\times$, with
\[
       \tau(v)=\zeta v
\]
for a primitive $n$-th root of unity $\zeta$.

A finite place of $F$ which ramifies in $E$ also ramifies in $L$.
Since $K/F$ is unramified at finite places, such a place must be
$x=a$ for some $a\in\F_q$. At this place $K/F$ splits completely,
while $L/K$ is totally ramified of degree $n$. Hence
\[
       e(L/F)=n.
\]
Since
\[
       e(L/F)=e(L/E)e(E/F),
\]
where $e(L/E)$ is a power of $p$ and $e(E/F)$ divides $n$, while
$p\nmid n$, it follows that
\[
       e(L/E)=1,\qquad e(E/F)=n.
\]
Kummer theory then gives
\[
       \gcd\bigl(n,v_{x=a}(\rho)\bigr)=1
       \qquad(a\in\F_q).
\]

For $a\in\F_q$, let $\sigma_a\in G$ be the automorphism above with
parameter $a$ and all other parameters equal to $0$. Conjugation by
$\sigma_a$ preserves the subgroup $U$. Hence any lift
$\widetilde\sigma_a\in\Gamma$ normalizes $\widehat U$ and therefore
its characteristic subgroup $U'$. Thus
\[
       \widetilde\sigma_a(E)=E.
\]
On $F$, the automorphism $\widetilde\sigma_a$ acts by $x\mapsto x+a$.
By \ref{it:central}, it commutes with $\tau$, so
$\widetilde\sigma_a(v)$ belongs to the $\zeta$-eigenspace $Fv$.
Therefore
\[
       \widetilde\sigma_a(v)=h_av
\]
for some $h_a\in F^\times$. Taking $n$-th powers gives
\[
       \rho(x+a)=h_a^{\,n}\rho(x),
       \qquad a\in\F_q.
\]

It follows that the valuations $v_{x=a}(\rho)$, $a\in\F_q$, are all
congruent modulo $n$ to the same residue class $u$, with
$\gcd(u,n)=1$. Choose $j$ such that
\[
       ju\equiv1\pmod n
\]
and replace $v$ by $v^j$. Since $\gcd(j,n)=1$, this is again a Kummer
generator of $E/F$. We may therefore assume that
\[
       v_{x=a}(\rho)\equiv1\pmod n
       \qquad(a\in\F_q).
\]
At every other finite place $\mathfrak p$ of $F$, the extension $E/F$
is unramified, and hence
\[
       v_{\mathfrak p}(\rho)\equiv0\pmod n.
\]
It follows that
\[
       \rho=\lambda\,(x^q-x)\,\psi^n
\]
for some $\lambda\in\F_Q^\times$ and $\psi\in F^\times$. Replacing
$v$ by $v/\psi$ gives
\[
       v^n=\lambda f.
\]
Finally, since $EK=L$, we have $L=K(v)$. This proves
\ref{it:normalform}.
\end{proof}

\begin{remark}\label{rem:normal-form}
For the argument below we only need the form
\[
       v^n=\lambda f,\qquad \lambda\in\F_Q^\times.
\]
The constant $\lambda$ cancels from the automorphism ratios used below,
so no further normalization is needed.
\end{remark}

\begin{remark}\label{rem:conductor}
In \cite[p.~536]{Skabelund} the field $L$ is described as the ray class
field ``of conductor $\mathfrak m$''. The inclusion
$K(t)\subseteq L$ from Proposition~\ref{prop:ray-structure}\ref{it:containment}
shows directly that the conductor of $L/K$ is exactly $\mathfrak m$.
Indeed, the subextension $K(t)/K$ is totally (and tamely) ramified at
every point of $D$: at the affine points $v_P(f)=1$, while at
$P_\infty$ one has $v_{P_\infty}(f)=1-N$ and $\gcd(m,N-1)=1$.
Since the conductor of $L/K$ divides $D$, it must therefore equal
$D=\mathfrak m$. Hence the two readings of the definition agree.
\end{remark}

\section{The standard involution and the forced non-gap}\label{sec:involution}

For both the Suzuki and Ree curves there is a standard
$\F_q$-involution $\iota$ which exchanges $P_\infty$ with an affine
rational point $P_0$; see \cite{HansenStichtenoth,Henn,Pedersen} and
\cite[Lemmas~3.3 and~4.2]{Skabelund}. We shall only use this action on
the base curve.

\begin{proposition}\label{prop:forced-nongap}
Let $n=[L:K]$. Then
\begin{equation}\label{eq:ndivN}
       n\mid N
       \qquad\text{and}\qquad
       \frac Nn\in H_X(P_\infty).
\end{equation}
Writing $n=mk$, which is possible by
Proposition~\ref{prop:ray-structure}\ref{it:containment}, one has
\begin{equation}\label{eq:forced-nongap}
       k\mid s_0,
       \qquad
       \frac{s_0}{k}=\frac Nn\in H_X(P_\infty).
\end{equation}
\end{proposition}

\begin{proof}
Let $\lambda$ and $v$ be as in \eqref{eq:twisted-normal-form}, and
choose a lift of $\iota$ to $L$, again denoted by $\iota$. By
Proposition~\ref{prop:ray-structure}\ref{it:central}, this lift
centralizes $A=\Gal(L/K)$. Let $\tau$ generate $A$ and write
\[
       \tau(v)=\zeta v,
\]
where $\zeta\in\F_Q^\times$ is a primitive $n$-th root of unity. Since
$\iota$ fixes $\F_Q$ pointwise,
\[
       \tau(\iota(v))
       =\iota(\tau(v))
       =\iota(\zeta v)
       =\zeta\,\iota(v).
\]
The $\zeta$-eigenspace of $\tau$ on $L$ is the one-dimensional
$K$-subspace $Kv$. Hence
\[
       \iota(v)=hv
\]
for some $h\in K^\times$. Taking $n$-th powers in
\[
       v^n=\lambda f
\]
and using $\iota(\lambda)=\lambda$ gives
\begin{equation}\label{eq:iota-ratio}
       \frac{\iota(f)}{f}=h^n.
\end{equation}

By Lemma~\ref{lem:divf},
\[
       \Div(f)=D-NP_\infty.
\]
Since $\iota$ permutes $X(\F_q)$ and
$\iota(P_\infty)=P_0$, we have
\[
       \Div(\iota(f))=D-NP_0.
\]
Taking divisors in \eqref{eq:iota-ratio} gives
\begin{equation}\label{eq:divh}
       n\Div(h)=N(P_\infty-P_0).
\end{equation}
Comparing coefficients at $P_\infty$ gives $n\mid N$, and
\[
       \Div(h^{-1})
       =\frac Nn(P_0-P_\infty).
\]
Thus $h^{-1}$ has a unique pole at $P_\infty$, of order $N/n$, on
$X_{\F_Q}$. By Lemma~\ref{lem:basechange-semigroup}, the same integer
$N/n$ is a non-gap at $P_\infty$ on $X/\F_q$. This proves
\eqref{eq:ndivN}.

Finally, $N=ms_0$ and $n=mk$, so $k\mid s_0$ and
\[
       \frac Nn=\frac{s_0}{k}.
\]
This proves \eqref{eq:forced-nongap}.
\end{proof}

\section{Proof of the main theorem}\label{sec:proof}

The first positive non-gap at $P_\infty$ is $q$ for $\Sm_q$ and $q^2$
for $\Rm_q$.  Indeed, the pole orders \eqref{eq:suzuki-x-pole} and
\eqref{eq:ree-x-pole} show that $q$ and $q^2$, respectively, are
non-gaps.  If $a$ denotes the first positive non-gap,
Lemma~\ref{lem:lewittes} gives
\[
       q^2+1\le qa+1
       \quad\text{for }\Sm_q,
       \qquad
       q^3+1\le qa+1
       \quad\text{for }\Rm_q.
\]
Hence $a\ge q$ and $a\ge q^2$, respectively.

\begin{proof}[Proof of Theorem~\ref{thm:main}]
Write $n=mk$ as in Proposition~\ref{prop:forced-nongap}.  Suppose first
that $X=\Sm_q$.  From \eqref{eq:s0-values},
\[
       s_0=q+2q_0+1.
\]
If $k>1$, Proposition~\ref{prop:forced-nongap} gives the positive
non-gap $s_0/k$ at $P_\infty$.  But
\[
       \frac{s_0}{k}
          \le\frac{q+2q_0+1}{2}
          <q,
\]
because $q=2q_0^2$ and $q_0\ge2$, so that $q-2q_0-1=2q_0^2-2q_0-1>0$.
This contradicts the fact that the first positive non-gap at $P_\infty$
is $q$.  Hence $k=1$.

Now let $X=\Rm_q$.  In this case
\[
       s_0=q^2+3qq_0+2q+3q_0+1.
\]
If $k>1$, then again $s_0/k$ is a positive non-gap at $P_\infty$, while
\[
       \frac{s_0}{k}\le\frac{s_0}{2}<q^2.
\]
Indeed, using $q=3q_0^2$,
\begin{align*}
 2q^2-s_0
 &=q^2-3qq_0-2q-3q_0-1\\
 &=9q_0^4-9q_0^3-6q_0^2-3q_0-1>0,
\end{align*}
since $9q_0^3(q_0-1)\ge18q_0^3>6q_0^2+3q_0+1$ for $q_0\ge3$.
This contradicts the fact that the first positive
non-gap at $P_\infty$ is $q^2$.  Thus $k=1$ here as well.

In both cases $n=m$.  By
Proposition~\ref{prop:ray-structure}\ref{it:containment}, the Skabelund
field $K(t)$, which has degree $m$ over $K$, is contained in $L$.  The
two extensions of $K$ now have the same degree, so $L=K(t)$.  Hence
\[
       \Sm_{\rm rcf}=\tSm_q,
       \qquad
       \Rm_{\rm rcf}=\tRm_q,
\]
in the respective cases.
\end{proof}

\section{A remark on maximal Kummer quotients}\label{sec:kummer-quotients}

For an integer $r$ prime to the characteristic, let $C_r$ denote the
smooth projective model of the affine Kummer curve
\[
       C_r:\qquad u^r=x^q-x
\]
over the rational $x$-line.  Skabelund already observed that in the
Suzuki case $C_{q^2+1}$ is maximal over $\F_{q^4}$, although $q^2+1$ is
a proper multiple of $m$, and he noted that proper multiples also occur
in the Ree case; see the discussion following
\cite[Cor.~5.5]{Skabelund}.  Thus maximality of $C_r$ alone cannot
determine the degree of the ray class extension.  For completeness, we
record an explicit uniform Ree example.

\begin{proposition}\label{prop:larger-maximal-kummer}
In the Ree case, the curve $C_{q^2-q+1}$ is maximal over $\F_{q^6}$,
and $q^2-q+1$ is a proper multiple of $m=q-3q_0+1$.
\end{proposition}

\begin{proof}
Put
\[
       r=q^2-q+1=(q+1)^2-9q_0^2=m(q+3q_0+1)>m.
\]
Since $\F_{q^2}^\times$ is cyclic and $q$ is odd, if $\xi$ is a
generator then $\epsilon=\xi^{(q+1)/2}$ satisfies
$\epsilon^{q-1}=\xi^{(q^2-1)/2}=-1$.  Fix such an
$\epsilon\in\F_{q^2}$.  The substitution $x=\epsilon x'$ transforms
$x^q-x$ into
\[
       \epsilon^qx'^q-\epsilon x'=-\epsilon(x'^q+x').
\]
Since
\[
       \frac{q^6-1}{r}=(q+1)(q^3-1)=(q^2-1)(q^2+q+1),
\]
every element of $\F_{q^2}^\times$, in particular $-\epsilon$, is an
$r$-th power in $\F_{q^6}$.  Thus, over $\F_{q^6}$, the curve $C_r$ is
isomorphic to the smooth projective model of
\[
       u^r=x^q+x.
\]
It is a quotient of the smooth projective model $W$ of
\[
       w^{q^3+1}=x^q+x,
\]
via $u=w^{q+1}$, because $q^3+1=(q+1)r$.

We claim that $W$ is maximal over $\F_{q^6}$.  The $\F_q$-linear map
\[
       T:\F_{q^6}\longrightarrow\F_{q^6},\qquad T(x)=x^q+x,
\]
has kernel of size $q$.  Its restriction to $\F_{q^3}$ maps $\F_{q^3}$
to itself and is injective: if $x\in\F_{q^3}$ and $x^q=-x$, then
$x^{q^2}=x$, so $x\in\F_{q^2}\cap\F_{q^3}=\F_q$, and hence $x=-x$,
which in characteristic $3$ forces $x=0$.  Thus $T$ maps $\F_{q^3}$
bijectively onto $\F_{q^3}$, and
\[
       T^{-1}(\F_{q^3})=\F_{q^3}+\ker T
\]
has $q^4$ elements, exactly $q$ of which satisfy $T(x)=0$.  Every
nonzero element of $\F_{q^3}$ has exactly $q^3+1$ roots of order
$q^3+1$ in $\F_{q^6}$, while elements of $\F_{q^6}\setminus\F_{q^3}$
have none.  Since $\gcd(q^3+1,q)=1$, the curve $W$ has a unique point
at infinity, which is rational.  Hence
\[
 \#W(\F_{q^6})
   =q+(q^4-q)(q^3+1)+1
   =q^7+1.
\]
The $q$ roots of $x^q+x$ and the point at infinity are totally
ramified in $W\to\PP^1_x$, so Riemann--Hurwitz gives
\[
       g(W)=\frac{(q-1)q^3}{2},
\]
and $\#W(\F_{q^6})=q^6+1+2g(W)q^3$ is the Hasse--Weil upper bound.
Hence $W$ is maximal over $\F_{q^6}$.  Since $C_r$ is covered by $W$
over $\F_{q^6}$, it is maximal as well by the covering result of Serre;
see \cite{Lachaud}.
\end{proof}

The point is that maximality of an auxiliary Kummer quotient is weaker
than the ray-class structure.  In the proof of Theorem~\ref{thm:main},
the additional input is the invariance under the base automorphisms and
the resulting central action on the ray class field.  The standard
involution then yields the divisor identity \eqref{eq:divh}, and the
first positive non-gap excludes every proper multiple of $m$.

\section*{Acknowledgments}

The author was partially supported by CNPq grant no.~302774/2025-4,
FAEPEX grant no.~3485/25, and FAPESP grant no.~2024/00923-6.

\end{document}